\documentclass[reqno]{amsart}
\usepackage{amsrefs}
\address{Columbia University, New York, NY, 10027, USA}
\address{Uppsala University, Department of Mathematics, P.O. Box 480,751 06 Uppsala, Sweden}

\usepackage[all]{xy}
\usepackage{mathtools}
\usepackage{euscript}

\theoremstyle{theorem}
\newtheorem{thm}{Theorem}[section]
\newtheorem{cor}[thm]{Corollary}
\newtheorem{lem}[thm]{Lemma}
\newtheorem{prop}[thm]{Proposition}
\newtheorem{conj}[thm]{Conjecture}

\theoremstyle{definition}
\newtheorem{defn}[thm]{Definition}

\DeclareMathOperator{\im}{im}
\DeclareMathOperator*{\colim}{colim}
\newcommand{\R}{\mathbb{R}}
\newcommand{\C}{\mathbb{C}}
\newcommand{\Z}{\mathbb{Z}}
\newcommand{\loops}{\mathcal{L}}

\newcommand{\loopr}{\loops_r}
\newcommand{\lT}{\loops T^*}
\newcommand{\Tlr}{\loopr T^*}

\newcommand{\TN}{T^*N}
\newcommand{\TL}{T^*L}
\newcommand{\Tuo}{\mathcal{T}}
\newcommand{\Tuoloop}{\mathcal{T}_{\mathcal{L}}}
\newcommand{\Tuoloopzero}{\mathcal{T}_{\mathcal{L}_0}}
\newcommand{\LLCW}{\loops_{\textrm{\tiny{CW}}} L}
\DeclareMathOperator{\Th}{Th}
\DeclareMathOperator{\U}{U}
\DeclareMathOperator{\Sp}{Sp}
\DeclareMathOperator{\Or}{O}
\newcommand{\UO}[1]{\raisebox{0.01em}{$\U#1$}\hspace{-0.14em}\raisebox{0.02em}{$/$}\hspace{-0.10em}\raisebox{-0.02em}{$\Or#1$}}
\newcommand{\uo}{\raisebox{0.01em}{$\U$}\hspace{-0.14em}\raisebox{0.02em}{$/$}\hspace{-0.10em}\raisebox{-0.02em}{$\Or$}}

\DeclareMathOperator{\id}{id}
\newcommand{\finr}{S_r}

\DeclarePairedDelimiter\norm{\lVert}{\rVert}

\DeclarePairedDelimiter\sqbrac{[}{]}
\DeclarePairedDelimiter\pare{(}{)}

\numberwithin{equation}{section}
\def\co{\colon\thinspace}
\begin{document}

\title[On the immersion classes of nearby Lagrangians]{On the immersion classes of nearby Lagrangians  }

\author[M. Abouzaid, T. Kragh]{Mohammed Abouzaid, Thomas Kragh}
\thanks{This research was performed during the time the first author served as a Clay Research Fellow and the second author was funded by the Carlsberg Foundation. The first author was also partially supported by NSF grant DMS-1308179}

\begin{abstract} We show that the transfer map on Floer homotopy types associated to an exact Lagrangian embedding is an equivalence. This provides an obstruction to representing isotopy classes of Lagrangian immersions by Lagrangian embeddings, which, unlike previous obstructions, is sensitive to information that cannot be detected by Floer cochains.  We show this by providing a concrete computation in the case of spheres.
\end{abstract}

\date{\today}


\maketitle
\setcounter{tocdepth}{1}
%\tableofcontents

\section{Introduction}
\label{sec:intro}

Let $N$ be a closed smooth manifold of dimension $n$.  Its cotangent bundle $\TN$, equipped with the canonical Liouville form, is an exact symplectic manifold.  Let $L \subset \TN$ be a closed exact Lagrangian submanifold.  Much of the study of Lagrangian embeddings in the last $25$ years has centered on the following conjecture of Arnol'd: 
\begin{conj}[Nearby Lagrangian Conjecture]
  $L$ is Hamiltonian isotopic to the zero section of $\TN$.
\end{conj}

The outcome of recent progress is that the inclusion of $L$ is necessarily a homotopy equivalence, with $L$ having, in addition, the property that its Maslov class vanishes (see \cites{kragh-nearby,abouzaid-nearby}).  In a different direction, there are special cases, see \cites{abouzaid-exotic,ekholm-smith}, in which we now know that $L$ is diffeomorphic to the zero section. The purpose of this paper is to study a different direction, by finding obstructions to the existence of embedded representatives in a given class of Lagrangian immersions.

\subsection{Lagrangian immersions of spheres}
 We shall state a general result later in this introduction, but it is useful to first discuss the case of spheres. The aforementioned results imply that it suffices to study  immersions in the homotopy class of the zero section. As we discuss in Section \ref{sec:appl-lagr-embedd}, Gromov's $h$-principle implies that classes of such immersions correspond to the homotopy groups of the unitary group. Bott periodicity can be used to compute these groups:
\begin{equation}
  \pi_{n}(U(n)) = \begin{cases} 0 & \textrm{ if $n$ is even} \\
\Z & \textrm{ if $n$ is odd.} 
\end{cases}
\end{equation}
This allows us to restrict attention to odd dimensional spheres, where the immersion class is not unique. Before stating our result for Lagrangian embeddings, we recall what is known in the smooth category:
\begin{lem} \label{lem:audin}
If $n \neq 1,3$, every Lagrangian immersion of a sphere in the homotopy class of the $0$ section of $T^{*} S^{n} $ is regulary homotopic to a smooth embedding. 
\end{lem}
\begin{proof}
By the Whitney trick, it suffices to show that such an immersion $j$ has an even number of double points. Denote by $\bar{j}$ the composition with the immersion of a neighbourhood of the $0$-section in $\C^{n}$, associated to the Whitney immersion $S^{n} \to \C^{n}$ with one double point. By \cite[Proposition 0.3]{Audin} $\bar{j}$ has an odd number of double points.    Because the degree of the composition of $j$ with the projection to $S^n$ is odd, the parities of the number of self-intersections of $j$ and $\bar{j}$  are different  (see, e.g. \cite[Proposition 5.1.1]{Audin}). Hence $j$ has an even number of double points, which completes the proof.
\end{proof}
Applying the results of this paper to spheres, we shall obtain the following result in Section \ref{sec:appl-lagr-embedd}, which should be compared to the above:
\begin{thm} \label{thm:main-application}
Whenever $n$ equals $1$, $3$, or $5$ modulo $8$, there is a class of Lagrangian immersions of $S^{n}$ in $T^{*} S^{n}$, in the homotopy class of the $0$ section, which does not admit an embedded representative.
\end{thm}
 We shall see in Section \ref{sec:appl-lagr-embedd} that we cannot address the case $7$ modulo $8$ because the map
 \begin{equation}
   \pi_{8k+7}(U) \to \pi_{8k+7}(U/O)
 \end{equation}
vanishes. Also, note that the nearby Lagrangian conjecture would imply the stronger statement that, with the exception of reparametrisations of the zero section by diffeomorphisms, no class of Lagrangian immersions admits a representative which is a Lagrangian embedding.

\subsection{Equivalence of Floer homotopy types}
\label{sec:equiv-floer-homot}

To state the main result of this paper, consider the pullback of the tangent bundle of $N$ by the evaluation map (at the base point) from the component $\loops_{0} N$ of the free loop space consisting of contractible loops.  Let $\loops_{0} N^{-TN}$ denote the Thom spectrum associated to the virtual vector bundle $-TN$:  concretely, we choose a vector bundle $E$ of rank $k$ on $N$ whose direct sum with $TN$ is trivial of rank $k+n$, and define $ \loops_{0} N^{-TN} $ to be the formal shift by $-(k+n)$ of the suspension spectrum of the Thom space of the pullback of $E$ to $\loops_{0} N$.

Given an exact Lagrangian embedding $j \co L \to \TN$, the second author defined a virtual (Maslov) bundle $\eta$ on $\loops_{0} L$ whose construction is recalled in Section \ref{sec:appl-lagr-embedd}, and a restriction map
\begin{equation} \label{eq:homotopy_equivalence_old}
  \loops_{0} N^{-TN} \to \loops_{0} L^{-TL + \eta}
  %\loops_{0} N^{-TN} \stackrel{\loops j_{!}}{\to} \loops_{0} L^{-TL + \eta}
\end{equation}
which generalises, to Thom spectra, a transfer map due to Viterbo.  It is important to note, at this stage, that the existence of this map is a consequence of \emph{symplectic} topology, and that the condition that $L$ be an embedded exact Lagrangian cannot be dropped for such a map to exist. If $N$ is orientable then by the Thom isomorphism theorem the homology of the left hand side agrees with the homology of $\loops_0 N$ shifted by the dimension of $N$.

The construction of this map is natural with respect to adding a virtual vector bundle to the loop space of $N$, which can be pulled back to the loop space of $L$ by the obvious map, and which by abuse of notation we denote by the same name. This means that one can add a copy of the tangent bundle $TN$ on both sides and get a map
\begin{align} \label{eq:homotopy_equivalence}
  \loops_0 j_! \co \Sigma^\infty \loops_{0} N_+ \to \loops_{0} L^{TN-TL + \eta}.
\end{align}
This was specifically noted in \cite[Corollary 14.4]{kragh-transfer}. Here $\Sigma^\infty(-)$ denotes the suspension spectrum, which is also the Thom spectrum defined using the trivial virtual bundle class.

The main result of this paper is the following theorem, proved in Section \ref{sec:homotopy_equivalence_of_spectra}:
\begin{thm} \label{thm:equivalence_spectra}
  The restriction map $\loops_0 j_{!}$ is a homotopy equivalence.
\end{thm}
In Proposition \ref{prop:hom_triv}, we show that $\loops_{0} L^{TN-TL + \eta}$ has the same homotopy type as $ \Sigma^\infty \loops_{0} N_+$ if and only if the classifying map for $\eta$ as a stable sphere bundle is null-homotopic. This gives a concrete criterion for finding an obstruction to $\loops_0 j_{!}$ being a homotopy equivalence, which, combined with classical computation of homotopy groups, readily yields the applications discussed in Theorem \ref{thm:main-application}.

The main technical result of this paper, used to prove Theorem \ref{thm:equivalence_spectra}, is that the transfer maps are compatible with the classical map of loop spaces
\begin{equation}
    \loops_{0} L \stackrel{\loops_{0} j}{\to} \loops_{0} N.
\end{equation}
To state the compatibility, we use the fact that the diagonal inclusion of a space $X$ into the product $X\times X$ gives rise to a natural map of Thom-spectra
\begin{align} \label{eq:comodule_structure}
  X^V \to X_+ \wedge X^V,
\end{align}
where $V$ is any virtual vector bundle over $X$, and the smash product of a space and a spectrum is easily defined level-wise (as opposed to smash products of spectra). The two cases of this that we are interested in are the two vertical maps in the following theorem, which will imply Theorem \ref{thm:equivalence_spectra}.
\begin{thm} \label{thm:main-technical-theorem}
  The following diagram commutes
  \begin{equation} \label{eq:commutative_diagram}
    \xymatrix{  \Sigma^\infty\loops_{0} N_+ \ar[rr]^{\loops_{0} j_{!}} \ar[rr]  \ar[d]& &   \loops_{0} L^{TN-TL + \eta}  \ar[d] \\ 
   \loops_{0} N_{+} \wedge \Sigma^\infty \loops_{0} N_+  \ar[r]^-{\id \wedge \loops_{0} j_{!}} &  \loops_{0} N_{+} \wedge    \loops_{0} L^{TN-TL + \eta}   & \loops_{0} L_{+} \wedge    \loops_{0} L^{TN-TL + \eta}    \ar[l]_-{\loops_{0} j \wedge \id} .  }  
  \end{equation}
\end{thm}

While we shall not prove this, the following fact may give the reader a useful perspective: $X_{+}$ is a coalgebra, and the map in Equation \eqref{eq:comodule_structure} makes $X^V$ into a comodule in the category of spectra.  With this in mind, Equation \eqref{eq:commutative_diagram} asserts that the restriction map $ \loops_{0} j_{!} $ is a map of comodules in the category of spectra. This structure is in a sense dual to the cup product on the cohomology of loop spaces used by Viterbo in \cite{viterbo-JDG}.

We end this introduction with an outline of the paper: Section \ref{sec:appl-lagr-embedd} provides the construction of the stable bundle $\eta$. We then state Proposition \ref{prop:hom_triv}, which is a weaker version of Theorem \ref{thm:equivalence_spectra}, that is then used to prove the applications to immersions of spheres. In Section \ref{sec:proof-prop-refpr}, we explain how  Proposition \ref{prop:hom_triv} follows from Theorem \ref{thm:equivalence_spectra}. The constructions of the second author in \cite{kragh-transfer,kragh-nearby} are summarised in Sections \ref{sec:construct_spectrum} and \ref{sec:transfer_map}, with a special focus on the results we shall need. Theorem \ref{thm:main-technical-theorem} is proved in Section \ref{sec:diagram_commutes}, and Theorem  \ref{thm:equivalence_spectra} is derived from it in the last section.

\subsection*{Acknowledgements}
We would like to thank Ivan Smith for suggesting Lemma \ref{lem:audin}.

% \section{Application to Lagrangian embeddings}
% label{sec:appl-lagr-embedd}
%\input{applications.tex}
\section{Application to Lagrangian embeddings}
\label{sec:appl-lagr-embedd}

In this section we will apply Theorem~\ref{thm:equivalence_spectra} to prove Theorem~\ref{thm:main-application} which is intended as a sample application. In fact, we will prove a much stronger result for a general  exact Lagrangian embedding $j \co L \to \TN$; we recall that such a map is a homotopy equivalence by the results of \cite{abouzaid-nearby,kragh-nearby}.

First, we recall the definition of the canonical (up to homotopy) map
\begin{equation} \label{eq:map-L-grassmannian}
  \Tuo \co L \to \uo = \colim_{K\to \infty} \UO{(K)}
\end{equation}
defined for any Lagrangian immersion $j\co L \to \TN$. Indeed, let $i\co N \subset \R^K$ be any embedding. Viewing $\R^K$ as a subset of $\C^K=\R^{2K}$ we get the hermitian bundle $\C^{K} \times N = i^*T\R^{2K}$ which splits canonically as
\begin{align}
  \label{eq:split}
  \C^{K} \times N \cong (\nu\otimes\C) \oplus (TN\otimes \C) \cong (\nu\otimes\C) \oplus T(T^*N)_{\mid N},
\end{align}
where $\nu \to N$ is the normal bundle of the embedding $i$, and $T^*N$ has the almost complex structure induced by the Riemannian structure on $N$ induced by the embedding $i$. The tangent bundle of $T^*N$ (as a hermitian bundle) is canonically isomorphic to the pull-back of the restriction $T(T^*N)_{\mid N}$. So for each point $q$ in $L$ (by taking the tangent bundle) we get a Lagrangian in $T_{j(q)}(T^*N)$, which by direct sum with the canonical Lagrangian $\nu \subset \nu \otimes \C$ and the splitting above defines a Lagrangian in $\C^K$. This defines a map $L \to \UO{(K)}$ (the Grassmannian of Lagrangian sub-spaces in $\C^K$). Since everything is continuous in the embedding $i$, and all embeddings become isotopic after increasing $K$ this defines $\Tuo$ up to homotopy.

Next, we pass to the free loop space, and consider the map 
\begin{equation}
\loops L \xrightarrow{\loops \Tuo} \loops \uo 
\end{equation}
induced by $\Tuo$.
Since $\uo \simeq \Omega^6 \Or$ is a loop space by real Bott periodicity, we have a retraction
\begin{equation}
  p \co \loops \uo \to \Omega \uo.
\end{equation}
Bott periodicity also identifies $\Omega \uo $ with $\Z \times B\Or$. We can now define the virtual bundle $\eta$ appearing in Equation~\eqref{eq:homotopy_equivalence}.
\begin{defn}
  The complement of the Maslov bundle $-\eta$ associated to a Lagrangian immersion $j$ is the virtual real vector bundle over the free loop space of $L$ classified by the map
\begin{align}\label{Jmap}
  \Tuoloop \co \loops L \xrightarrow{\loops \Tuo} \loops \uo \xrightarrow{p} \Omega \uo \simeq \Z \times B\Or.
\end{align}
\end{defn}
We denote by $  \Tuoloopzero$  the restriction of $ \Tuoloop  $ to the component of the loop space consisting of contractible loops.

The next ingredient is the canonical map from $\Or$ to $H$ where $H$ is the stable group of self-homotopy equivalences of spheres, which is called the $J$-homomorphism. \emph{A} canonical delooping of this map
\begin{equation} \label{eq:BJ}
  BJ \co  \Z \times B\Or \to BH
\end{equation}
is given by associating to a virtual vector bundle the corresponding stable sphere bundle (and forgetting its dimension), c.f. \cite{Atiyah}. However, be warned that there is a difference in that paper between the map $J$ and the $J$-homomorphism. In fact, the $J$ there is induced by our $BJ$. The homotopy groups of $BH$ coincide with the stable homotopy groups of spheres shifted up by one except in degree 1 where it is only the units $\pm 1 \subset \Z$, see  \cite[Lemma 1.3]{Atiyah}.

The precise result we shall prove in the next Section is:
\begin{prop}
  \label{prop:hom_triv}
  If $j\co L \to \TN$ is an exact Lagrangian embedding then the composition $BJ \circ \Tuoloopzero$ is null-homotopic,  i.e. the stable sphere bundle associated to $\eta$ is trivial on $\loops_{0} L$.
\end{prop}

The reader may think of the above result as a more concrete version of Theorem \ref{thm:equivalence_spectra}. In order to use it, we recall the statement of the $h$-principle for Lagrangian immersions \cite{gromov-pdr,Lees}: the set of Lagrangian regular homotopy classes homotopic to a given map
\begin{equation}
i_{0} \co  L \to M
\end{equation}
from a smooth manifold $L$ of dimension $n$, to an exact  symplectic manifold $M$ of dimension $2n$, can be identified with the connected components of the space of maps of vector bundles
\begin{align}
  T^*L \to i^{*}_{0} TM 
\end{align}
which have Lagrangian image. Letting $\Sp(TM)$ denote the bundle over $M$ with fibre over $x\in M$ the Lie group of linear automorphisms of  $T_xM$ preserving the symplectic form, we note that the space of such maps of vector bundles is either empty, or a principal homogeneous space (torsor) over the space of sections
\begin{equation}
  L \to i^{*}_{0} \Sp(TM),
\end{equation}
i.e. the space of sections acts transitively and freely.

Let us specialise this to the case 
\begin{equation} \label{eq:assumption_parallelisable}
\parbox{30em}{$L = N$ has tangent bundle whose complexification is trivial, and $i_0$ is the inclusion of the $0$-section. }
\end{equation}
 In this case, $\Sp(TM)$ is the trivial fibre bundle $\Sp(n) \times N$.  Moreover, since the inclusion of $\U(n)$ into $\Sp(n)$ is a homotopy equivalence and the inclusion of $U(n)$ into $\U$ is $2n-1$ connected, we conclude the following.

\begin{lem}
Under Assumption \eqref{eq:assumption_parallelisable}, equivalence classes of Lagrangian immersions in the homotopy class of $i_0$ are classified by homotopy classes of maps from $L$ to $\U$.\qed
\end{lem}
\begin{cor}
  Isotopy classes of Lagrangian immersions of the sphere $T^*S^{n}$, in the homotopy class of the standard embedding, are classified by $\pi_{n}(\U)$. \qed
\end{cor}

From now on, we shall write $j$ either for an isotopy class of Lagrangian immersions, or for the associated homotopy class of maps to $\U$. Consider the quotient
\begin{equation}
  \label{eq:quotient-U-UO}
  q \co \U \to \uo.
\end{equation}
By composing $j$ with $q$ we recover the map from Equation \eqref{eq:map-L-grassmannian}
\begin{align}
  \label{eq:loopuomap}
  \Tuo \co L \to \uo.
\end{align}
Recalling that $\Z \times B\Or $ classifies virtual \emph{real} vector bundles, and the particulars of Bott periodicity, we get when looping $q$ the map
\begin{equation}
  \label{eq:loop-of-q}
  \Omega q \co \Z \times B\U \to \Z \times B\Or
\end{equation}
induced by the map on virtual vector bundles given by forgetting the complex structure.

\begin{proof}[Proof of Theorem~\ref{thm:main-application}]
  By Bott periodicity we have
  \begin{equation}
    \label{eq:Bott-periodicity}
    \pi_{n}(\U) \cong \pi_{n-1}(\Omega \U) \cong \pi_{n-1}(\Z \times B\U),
  \end{equation}
  and we let $j'$ denote a representative of $j$ under this isomorphism. Now consider the commutative diagram
\begin{equation}
  \label{eq:eta_class_diagram}
  \xymatrix{
    \loops S^{n} \ar[r]^{\loops{j}} & \loops \U \ar[r]^{\loops q} & \loops \uo \ar[r]^{p} & \Omega \uo \ar[r]^-{BJ} & BH \\
    \Omega S^{n} \ar[u] \ar[r]^{\Omega j} & \Omega \U \ar[u] \ar[r]^-{\Omega q} & \Omega \uo \ar[u] \ar[r]^-{=} & \Z \times B\Or \ar[u]^{=} \\
    S^{n-1} \ar[u]^i \ar[r]^{j'} & \Z \times B\U \ar[u]^{\simeq},
  }
\end{equation}
where $i$ is the adjoint to the homotopy equivalence $\Sigma S^{n-1} \to S^n$, and the first row of vertical maps are the usual inclusions. By Proposition~\ref{prop:hom_triv} the composition of the horizontal maps in the first row is homotopically trivial. So, we conclude that the composition
\begin{align}
  \label{eq:trivial_composition}
  S^{n-1} \xrightarrow{j'} \Z\times B\U \xrightarrow{\Omega q} \Z \times B\Or \xrightarrow{BJ} BH 
\end{align}
is null-homotopic.

  In Table \ref{table1}, we list the maps induced by $BJ$, $\Omega q$, and their composition on homotopy groups. Here $k=n-1$.
  \begin{table}
    \begin{tabular}{|l|cccccccc|}
      \hline 
      $k \mod 8$       & 0 & 1   & 2 & 3  & 4 &  5 & 6& 7 \\
      \hline
      $\pi_k(\Z\times B\U)$       &$\Z$ & 0 &$\Z$& 0 &$\Z$& 0&$\Z$ &0 \\
      $\pi_k(\Z\times B\Or)$   &$\Z$   &$\Z/2$&$\Z/2$&0&$\Z$& 0& 0 &0 \\
      $\im((\Omega q)_*)$        &$2\Z$ & 0 & $\Z/2$& 0 &$\Z$& 0& 0 &0\\
      $\im((BJ)_*)$     &$\Z/c_k$   & $\Z/2$&$\Z/2$&0&$\Z/c_k$&0&0&0 \\
      $\im(( BJ \circ \Omega q)_*)$ &$\Z/ c'_k $&0&$\Z/2$&0&$\Z/c_k $&0&0&0 \\
      \hline
    \end{tabular}
    \caption{Maps on homotopy groups for $k \geq 1$.}
    \label{table1}
  \end{table}
  The image of $q$ (and thus $\Omega q$) is easily found using the long exact sequence for $\Or \to \U \to \uo$. The image of $BJ$ was computed by Adams and Quillen, see \cite[Theorem 1.1.13]{Ravenel}, and the image of the composition follows from the first two. In the last two rows, $c_k$ is the denominator of the reduced fraction of $B_{2k}/4k$, where $B_i$ is the $i^\textrm{th}$ Bernoulli number, and $c'_{k} = c_{k}/2$ (with the conventions $c_0=0$ and $\Z/c_0'=2\Z$ the table actually holds for all $k\geq 0$).

Because the entries in the last row for $n=1,3,5$ mod $8$ are non-zero, there are immersion classes whose image under $ ( BJ \circ \Omega q)_* $ does not vanish.   Proposition~\ref{prop:hom_triv} tells us that such classes  do not admit an embedded representative. 
\end{proof}

\section{Proof of Proposition \ref{prop:hom_triv}}
\label{sec:proof-prop-refpr}

In this section, we derive Proposition \ref{prop:hom_triv} from 
Theorem \ref{thm:equivalence_spectra}. The idea behind the proof is that if a Thom spectrum is homotopy equivalent to a suspension spectrum  then the corresponding classifying map to $BH$ is null homotopic. Some care will be needed to address the fact that the free loop space of a manifold does not in general have the homotopy type of a finite CW complex.

We know that $L \to N$ is a homotopy equivalence. In \cite{Atiyah}, Atiyah proved that, for a homotopy equivalence $L \to N$, the stable sphere bundle of $TN-TL$ is trivial. The essential ingredient of the proof is  \cite[Proposition 2.8]{Atiyah}, which is due to Milnor and Spanier. For completeness, we shall presently give the proof of this result using more modern language.

To this end,  introduce the notion of a \emph{co-reduction} of a spectrum $X$ (with $H_0(X)\cong \Z$) which is a map from $X$ to the sphere spectrum $\Sigma^\infty S^0$, which induces an isomorphism on homology groups in non-positive degrees. Notice that since the sphere spectrum has trivial homology in negative degrees the same has to be true for $X$ - this is called a connective spectrum. Any suspension spectrum $\Sigma^\infty Y_+$ of a connected space $Y$ has a co-reduction by taking the suspension spectrum functor of the based map $Y_+ \to S^0$ given by collapsing all of $Y$ to the non-base point.

We will state the following lemma in slightly more generality than we need. It is standard that there is a Thom space associated to any sphere bundle, and thus a Thom spectrum associated to any stable sphere bundle (see e.g. \cite[Definition 5.3]{Rudyak}).

\begin{lem}[Proposition 2.8 of \cite{Atiyah}] \label{lem:co_reduction_trivial}
Let $Y$ be a connected finite $CW$ complex, and $V$ a $0$ dimensional stable spherical bundle over $Y$. The Thom spectrum of $V$  admits a coreduction if and only if $V$ is trivial, i.e. if and only if the classifying map
\begin{equation}
  Y \to BH
\end{equation}
is null homotopic.
  \end{lem}
  \begin{proof}
    Since $Y$ is finite the co-reduction of $Y^{V}$ is realized as an actual map
\begin{align}
  \Th(V \oplus \epsilon^a) \to S^a,
\end{align}
where $\epsilon^a$ is the trivial bundle of dimension $a$ and $V \oplus \epsilon^a$ is represented by an actual bundle over $Y$ of which we take the Thom space. Since the Thom space is the quotient of the disc bundle by the sphere bundle we may lift this to a fibre-wise map over $Y$
\begin{align}
  S^+(V \oplus \epsilon^a) \to S^a \times Y,
\end{align}
where $S^+$ is the sphere bundle over $Y$ obtained as the fibre-wise quotient of the disc with the sphere over $Y$. Since the original map was a co-reduction it follows that the composition
\begin{align}
  S^a \to S^+(V \oplus \epsilon^a) \to S^a \times Y \to S^a  
\end{align}
is a homology equivalence and hence a homotopy equivalence. Here the first map is the inclusion of an abitrary fibre of the sphere bundle. This implies that the sphere bundle $S^+(V \oplus \epsilon^a)$ is trivial, which is equivalent to the homotopy triviality of the classifying map $Y \to BH$.
  \end{proof}

Atiyah's result that the stable sphere bundle of $TN-TL$ is trivial whenever $L \to N$ is a homotopy equivalence then follows immediately by fixing an embedding $L \to \R^{K}$, and noting that the Thom collapse map associated to the product map $L \to \R^{K} \times N$ induces a homotopy equivalence   $\Sigma^{\infty}N_+  \to L^{TN-TL}$. Together with Theorem \ref{thm:equivalence_spectra}, this implies that we have a homotopy equivalence
\begin{equation} \label{eq:coreduction}
(\loops_{0} L)^{\eta} \simeq \Sigma^\infty (\loops_{0} N)_+.
\end{equation}

Note that this yields a coreduction of $(\loops_{0} L)^{\eta}  $, hence we could apply Lemma \ref{lem:co_reduction_trivial} to conclude that $(\loops_{0} L)^{\eta}  $ is a suspension spectrum if  $\loops_0 L$ had the homotopy type of a finite $CW$ complex. Since this is not the case in general, we obtain the weaker result that the classifying map for $\eta$ is a \emph{finitely phantom map}, i.e. that the restriction to finite sub-complexes is null homotopic. Section 2.1 of \cite{MayPonto} provides an account of the theory of phantom maps, but the reader should be aware that there are two related notions, as explained in  Warning 2.1.13 in \cite{MayPonto}. The results we use apply for either notion:
\begin{lem} \label{lem:phantom_map_0}
  Let $X$ be a $CW$ complex, equipped with a filtration $X_0 \subset X_1 \subset \cdots \subset X_n \subset \cdots \subset X$ by finite $CW$ complexes, and $Y$ a space with finite homotopy groups in each dimension. If 
  \begin{equation}
    h \co X \to Y
  \end{equation}
is a map whose restriction to $X_i$ is null homotopic for every $i$, then $h$ is null homotopic.
\end{lem}
\begin{proof}
The finiteness of the homotopy groups of $Y$ implies that each of the groups $\sqbrac{\Sigma X_n,Y}$ is finite. This implies that $\lim^1\sqbrac{\Sigma X_n,Y}$ is trivial, and hence $h$ is null homotopic by Proposition 2.1.9 and Corollary 2.1.11 of \cite{MayPonto}.
\end{proof}

We have now gathered all the ingredients to prove our reformulation of Theorem \ref{thm:equivalence_spectra}:
\begin{proof}[Proof of Proposition \ref{prop:hom_triv}]
 Let $\LLCW \to \loops_{0} L$ be a CW approximation with countably many cells, which is a homotopy equivalence by \cite{Milnor}. Since $\loops_{0} L$ is connected we may assume that $\LLCW$ contains only one 0-cell. Define $h \co \LLCW \to BH$ as the composition
  \begin{align}
    \label{eq:h_definition}
    h\co \LLCW \simeq \loops_{0} L \xrightarrow{\Tuoloopzero} \Z\times B\Or \xrightarrow{BJ} BH,    
  \end{align}
  which we wish to prove is homotopic to a constant map.

By construction, $h$ classifies the pullback to $ \LLCW$ of the stable sphere bundle associated to $-\eta$, which we shall continue denoting $-\eta$ by abuse of notation. Equation \eqref{eq:coreduction} implies that the Thom space of $\eta$ over $\LLCW$ admits a coreduction. By Lemma \ref{lem:co_reduction_trivial}, we conclude that $h$ is a phantom map (i.e. its restriction to every finite subcomplex is null-homotopic). Lemma \ref{lem:phantom_map_0} then implies that $h$ is itself null-homotopic, hence that  $BJ \circ \Tuoloopzero$ is null-homotopic as well.
\end{proof}

% \section{Construction of the Floer homotopy type}
% \label{sec:construct_spectrum}
%\input{spectrum.tex}
\section{Construction of the Floer homotopy type}
\label{sec:construct_spectrum}

In this section we recall the construction in \cite{kragh-transfer} of spaces
that (are expected to) represent Hamiltonian Floer cohomology. Here
\emph{represent} is taken to mean that the homology of these 
spaces are naturally isomorphic to  Floer cohomology.  More precisely, these spaces represent a twisted version of Floer cohomology, as discussed in \cite{kragh-nearby}. Since we will
never need the relation of these spaces to Floer cohomology (except
as a useful analogy) the twisting is irrelevant to this paper.

Our notion of spectra is very naive, and we refer to \cite{Adams} for
an introduction to CW spectra. The only difference between these and the
spectra we use is that each level of our spectra is in fact a compact
Hausdorff space and the structure maps are only assumed to  be continuous. Some call these
pre-spectra since we do not assume that the structure maps are adjoint
to weak equivalences. Another slight difference is that we only define
every $(n+k)^\textrm{th}$ space and thus the structure maps are maps
$\Sigma^{n+k}X_{(n+k)r} \to X_{(n+k)(r+1)}$.

Let us fix a Riemannian metric on $N$, and write $\norm{p}$ for the
length of a cotangent vector.  We say that a Hamiltonian $H\colon \TN
\to \R$ is linear at infinity if, away from a compact set, it agrees
with
\begin{align}
  \mu \norm{p} + c
\end{align}
for a strictly positive constant $\mu$.  Later, we shall also assume that $\mu$ is not the length of any closed geodesic in $N$.

In \cite[Section 5]{kragh-transfer} a sequence of finite dimensional manifolds
$\Tlr N$ (denoted there by $T^*\Lambda_r N$, see \cite[Equation (18)]{kragh-transfer}) were defined as
approximations to the free loop space $\lT N$, together with a sequence of
functions 
\begin{align}
  \finr \colon \Tlr N \to \R
\end{align}
which approximate the action functional
\begin{align}
  A(\gamma)= \int \gamma^* \lambda - \int \gamma^*H dt.
\end{align}
These finite dimensional approximations admit pseudo-gradients $X_r$, which satisfy a Lipschitz like condition (to ensure that the flow is defined for all time, and gradient trajectories do not run of to infinity), implying that the negative gradient flows have well-defined Conley indices for $r$ sufficiently large by \cite[Lemma 5.4 and Lemma 2.9]{kragh-transfer}. In particular, we may consider an index pair $(A_r,B_r)$ for $(\finr,X_r)$ containing all the critical points of $\finr$ (see \cite{Conley} and \cite[Section 2]{kragh-transfer} for the definition of index pairs). 

By  \cite[Proposition 6.3]{kragh-transfer}, passing from $r$ to $r+1$ simply changes the
Conley index by a relative Thom space construction using the vector
bundle $TN$. That is, there are canonical (up to a contractible choice) homotopy equivalences
\begin{align} \label{eq:1}
  DTN_{\mid A_r} / (STN_{\mid A_r} \cup DTN_{\mid B_r}) \to A_{r+1} / B_{r+1}.
\end{align}
Here, by abuse of notation, we also denote by $TN$
the pull back of $TN$ to $A_r \subset \Tlr N \to N$, where the latter
map is the evaluation at the base-point composed with the projection to
$N$. 

To produce a spectrum from this sequence of pairs we need to
compensate for the fact that $TN$ is not a trivial bundle. So let
$\nu$ be a $k$-dimensional normal bundle to $TN$ (pulled back to
$A_r$), and assume we have a bundle isomorphism $TN\oplus \nu \cong
\R^{n+k} \times A_r$. Consider the vector bundle
\begin{align}\label{eq:2}
  \nu_r = \nu^{\oplus r} \to A_r
\end{align}
and define the $((n+k)r)^\textrm{th}$ space in the spectrum $Z(H^\mu)$ by
\begin{align}\label{eq:3}
  Z(H^\mu)_{(n+k)r} = D \nu_{r\mid A_r} / (S\nu_{r\mid A_r} \cup D\nu_{r\mid B_r}).
\end{align}
This is the \emph{Thom space} of $\nu_r \to A_r$ \emph{relative} to $B_r$, which gives rise to the left hand side of Equation~\eqref{eq:homotopy_equivalence} (see Remark 1.1 in \cite{kragh-transfer}).

The homotopy equivalences in Equation~\eqref{eq:1} turn into homotopy equivalences
\begin{align} \label{eq:4}
  \Sigma^{n+k} Z(H^\mu)_{(n+k)r} \simeq Z(H^\mu)_{(n+k)(r+1)}.
\end{align}
These where used as structure maps in defining a spectrum $Z(H^\mu)$. By \cite[Corollary 11.6]{kragh-transfer}, we have (for $\mu$ different from the length of a closed geodesic) a canonical (up to a contractible choice) homotopy equivalence
\begin{align}
  \label{eq:6}
  Z(H^\mu) \simeq \Sigma^{\infty} \loops^\mu N_+,
\end{align}
where the superscript $\mu$ means restriction to loops of length at most $\mu$.

By changing the Hamiltonian to have a higher slope $\mu'$ at infinity, one can produce continuation maps, which in the proof of Corollary 11.6 in \cite{kragh-transfer} was identified with the suspension spectrum functor on the usual inclusions of loop spaces $\loops^\mu N \subset \loops^{\mu'}
N$. Hence taking the colimit we canonically recover the spectrum
\begin{align}
  \colim_{\mu\to \infty} Z^{\mu} \simeq \Sigma^\infty \loops N_+.
\end{align}
Restricting this entire construction to the component of contractible loops in $\loops T^* N$ yields the spectrum $\Sigma^\infty \loops_0 N_+$.

% \section{Viterbo's transfer map on spectra}
% \label{sec:transfer_map}
%\input{transfer.tex}
\section{Viterbo's transfer map on spectra}
\label{sec:transfer_map}

In this section we recall the construction of the transfer map
\begin{align*}
  (\loops N)^{-TN} \to (\loops L)^{-TL+\eta},
\end{align*}
which when using the alternate construction explained in Remark 1.1 of \cite{kragh-transfer} becomes
\begin{align} \label{eq:11}
  \Sigma^\infty(\loops N)_+ \to (\loops L)^{TN-TL+\eta}.
\end{align}
Since the virtual bundle $TN-TL+\eta$ may not be representable on all of $\loops L$ as a difference of finite dimensional bundles, we have to use the following procedure to define the Thom spectrum: given a virtual bundle $V$ over a space $X$, pick an ascending filtration of $X$
\begin{align}
  K_1 \subset K_2 \subset \cdots \subset K_n \subset \cdots \subset X,
\end{align}
where each $K_i$ has the homotopy type of a finite CW complex, then realize the Thom spectrum as described in the introduction (using the virtual bundle $V$) on each of these, and then take the limit in the category of spectra using the inclusions.

% To define the target, we therefore take a filtration 
% \begin{align}
%   K_1 \subset K_2 \subset \cdots \subset K_n \subset \cdots \subset
%   \loops L, 
% \end{align}
% where each $K_i$ has the homotopy type of a finite CW complex, then realize the
% Thom spectrum as described in the introduction (using the virtual
% bundle $TN-TL+\eta$) on each of these, and then take the limit using the
% inclusions. 
% Indeed,  Indeed, when this is not the case the formal shift needed at each stage grows to infinity.
% This, however, is automatically contained in the following construction by increasing $r$.

In Section~\ref{sec:construct_spectrum}, we chose index pairs $(A_r,B_r)$ so that $A_{r}$  contains all critical points of
the function $\finr$. However, one could easily restrict to only consider critical points with critical value bounded below.  For $a\in
\R$, this defines a restricted index pair
\begin{align} \label{eq:10}
  (A_r',B_r') = \big((A_r \cap S_r^{-1}([a,\infty]) , (B_r \cap
  S_r^{-1}([a,\infty])) \cup (A_r \cap S_r^{-1}(a))\big)
\end{align}
which is the target of  a canonical quotient map
\begin{align}
  A_r /B_r \to A_r' / B_r'.
\end{align}
If $\zeta \to A_r$ is a vector bundle we can ``correct'' this map by
the relative Thom spaces construction to get a map
\begin{align}
  D\zeta / (S\zeta \cup D\zeta_{\mid B_r}) \to
  D\zeta_{\mid A_r'} / (S\zeta_{A_r'} \cup D\zeta_{\mid B_r'}).
\end{align}
Since we used such corrected Conley indices to define the source
spectrum in Section~\ref{sec:construct_spectrum} this is the way we
will construct the transfer map.

By the Darboux-Weinstein theorem, and by choosing a suitable Riemannian structure on
$L$ we can asssume that $DT^*L \subset T^*N$ is a symplectic
embedding. The usual way of producing the Viterbo transfer
(restriction) map is by choosing Hamiltonians as in
Section~\ref{sec:construct_spectrum}, whose restrictions to $DT^*L$ agree with
\begin{align}
\mu_L \norm{p_L} + c_L
\end{align}
for $0<\epsilon<\norm{p_L}<1-\epsilon$. If $\mu_L$ is not a geodesic length
on $L$ and large enough (depending on $\mu$) one can make sure that
all the critical action values coming from periodic orbits in
$D_\epsilon T^*L$ are larger than the rest of the critical values. We
thus define $(A'_r,B'_r)$ as above, so that $A'_r$ only contains these critical
points. The above quotient map on corrected Conley indices thus induces
a map
\begin{align}
  Z(H^\mu)_{(n+k)r} \to W(H^{\mu_L})_{(n+k)r},
\end{align}
where $W(H^{\mu_L})_{(n+k)r}$ is defined as the relative Thom space
construction of $\nu_r$ restricted to the index pair
$(A'_r,B'_r)$. The canonical homotopy equivalences from
Section~\ref{sec:construct_spectrum} restrict to homotopy
equivalences
\begin{align}
  \Sigma^{n+k} W(H^{\mu_L})_{(n+k)r} \simeq W(H^{\mu_L})_{(n+k)(r+1)},
\end{align}
and thus we get a spectrum $W(H^{\mu_L})$. Indeed, Lemma 8.1 in
\cite{kragh-transfer} states this compatibility. Again there are
continuation maps of spectra $W(H^{\mu_L}) \to W(H^{\mu'_L})$ for
$\mu_L'>\mu_L$, and moreover the quotient maps commute with these
(keeping $\mu_L$ large enough depending on $\mu$) to form a map of
spectra
\begin{align} \label{eq:9}
  \Sigma^\infty\loops N_+ \to \colim_{\mu_L \to \infty} W(H^{\mu_L})
\end{align}
as discussed in Section 14 of \cite{kragh-transfer}. At each stage $\mu_L$ there are homotopy equivalences
\begin{align}\label{eq:5}
  W(H^{\mu_L}) \simeq \loops^{\mu_L} L^{TN-TL+\eta},
\end{align}
where $\eta$ is the virtual
vector bundle over $\loops L$ described in
Section~\ref{sec:appl-lagr-embedd}. Furthermore, the continuation maps
are compatible with the Thom spectrum construction and are induced by the
inclusion of the space of shorter loops into the space of longer loops making
\begin{align} \label{eq:8}
  \colim_{\mu_L \to \infty} W(H^{\mu_L}) \simeq \loops L^{TN-TL+\eta}.
\end{align}
This is the result of Corollary 14.4 in \cite{kragh-transfer}. Because $L \to N$ is a homotopy equivalence, restricting the above construction to the component of contractible loops in $\loops T^* N$ yields the spectrum $\loops_{0} L^{TN-TL+\eta}$.

Combining Equation \eqref{eq:9}, \eqref{eq:5}, and \eqref{eq:8}, we obtain Equation \eqref{eq:11}. By only considering contractible loops, we obtain the desired map:
\begin{equation}
  \Sigma^\infty \loops_{0} N_+   \to \left( \loops_{0} L\right)^{TN-TL+\eta}.
\end{equation}
% \section{Proof of Theorem \ref{thm:main-technical-theorem}}
% \label{sec:diagram_commutes}
%\input{diagram_commutes.tex}

\section{Proof of Theorem \ref{thm:main-technical-theorem}}
\label{sec:diagram_commutes}

Let $(A,B)$ be any pair of spaces. There is a map
\begin{align}
  A/B \to A_+ \wedge A/B
\end{align}
induced by the diagonal $\Delta \colon A \to A \times
A$. Indeed, this is defined because $\Delta(B) \subset A\times B$ and
$A\times A / (A\times B) \cong A_+ \wedge A/B$.

Since the index pairs $(A_r,B_r)$ from
Section~\ref{sec:construct_spectrum} sit inside $\loops \TN $ (it is
constructed using finite dimensional approximations naturally embedded into
it) we obtain a composition
\begin{align}
  A_r/B_r \to A_{r+} \wedge A_r/B_r \to \loops \TN_+ \wedge  A_r/B_r . 
\end{align}
The correction done to the Conley index in Equation~\eqref{eq:3}
is easy to incorporate by using the projection from the disc bundle
$D\nu_{r\mid A_r}$ to $A_r$. In this way we may define a map
\begin{align}
  Z(H^\mu)_{(n+k)r} \to \loops \TN_+ \wedge Z(H^\mu)_{(n+k)r}.
\end{align}
These are compatible with the structure maps from
Equation~\eqref{eq:4} and thus we get a map of spectra
\begin{align}
  Z(H^\mu) \to \loops \TN_+ \wedge Z(H^\mu).
\end{align}
Here the smash product of a based space with a spectrum is simply
defined by the level wise smash product. Taking the limit over $\mu$ with respect to continuation maps \cite[Corollary 11.4]{kragh-transfer} yields that the induced map on spectra (defined only up to homotopy at this point) is the standard co-module map
\begin{equation}
   \Sigma^\infty \loops N_+ \to \loops N_+ \wedge \Sigma^\infty \loops N_+.
\end{equation}
We would like to have a similar construction for $\loops L$. For this we need the result from \cite{kragh-transfer} that; if you set up your Hamiltonians carefully (when defining the transfer map) then there are in fact index pairs $(A_r',B_r')$ defining each stage in the target spectrum in Equation~\eqref{eq:homotopy_equivalence}, which are completely contained inside $\loops_r DT^*L$ (i.e. a finite dimensional version of loops in this disc bundle). This is Proposition 10.1 with $M=DT^*L$ in that paper, and this is a topological version of the fact that flow lines from the critical points close to $L$ stay within $\loops D\TL$ (used in the proof of Viterbo functoriality). Using such index pairs we get for each $\mu_L$ an induced map of spectra
\begin{align}
  W(H^{\mu_L}) \to \loops \TL_+ \wedge W(H^{\mu_L}),
\end{align}
which by \cite[Corolary 14.5]{kragh-transfer} when we take the limit induce the standard ``co-algebra'' structure map (up to homotopy):
\begin{align*}
  (\loops L)^{TN-TL+\eta} \to \loops L_+ \wedge (\loops L)^{TN-TL+\eta}.
\end{align*}

%Since we replaced the pair defining the target of the map $A_r/B_r \to A_r'/B_r'$ it is convenient to mention that we still get such a quotient map by first restricting the pair as in Equation~\eqref{eq:10}, and then composing with the flow of the negative pseudo-gradient until the map is well-defined. Indeed, composing with the flow of the negative pseudo-gradient is how one proves that Conley indices are independent of the pair chosen (see \cite[Section 2]{kragh-transfer}).

The fact that the diagram in Equation~\eqref{eq:commutative_diagram} homotopy commutes now follows from the fact that the diagrams (before taking limits)
\begin{align*}
  \xymatrix{
    Z(H^\mu)  \ar[d] \ar[rr] & & W(H^{\mu_L}) \ar[d] \\
    \loops \TN_+ \wedge Z(H^\mu) \ar[r] & \loops \TN_+ 
    \wedge W(H^{\mu_L}) & \loops D\TL_+ \wedge  W(H^{\mu_L}) \ar[l] \\
  }
\end{align*}
homotopy commutes. Indeed; level-wise and before including $\loops_r T^*N$ into $\loops T^*N$, if we did not have to compose with the flow of the negative pseudo-gradient to have the quotient map well-defined this would commute on the nose. So, by composing the counterclockwise map with the same flow on the first factor provides a level-wise homotopy.

% \section{Proof of Proposition \ref{prop:equivalence_spectra}}
% \label{sec:homotopy_equivalence_of_spectra}
%\input{equivalence_spectra.tex}

\section{Proof of Theorem \ref{thm:equivalence_spectra}}
\label{sec:homotopy_equivalence_of_spectra}

Let $H\Z$ be the Eilenberg-Maclane ring spectrum for the ring $\Z$. We have a Thom isomorphism
\begin{align} \label{eq:13}
  \tau_L \colon  \loops_0 L^{TN-TL+\eta} \wedge H\Z & \to \Sigma^{\infty} \loops_0 L_+ \wedge H\Z.
\end{align}
Note that there is no shift here since the Maslov class vanishes on the contractible loops and thus $\dim(\eta)=0$. Also we have used the fact that $L\to N$ is relatively oriented and spin (because it is a homotopy equivalence) to conclude that $\eta$ and $TN-TL$ are oriented vector bundles (see \cite[Corollary 7.5]{kragh-nearby}).

The Thom isomorphism is induced by taking the product with the Thom class, which corresponds to a map
\begin{align}
  p_L & \colon \loops_0 L^{TN-TL+\eta} \to H\Z
\end{align}
which represents a generator of the zeroth cohomology of $  \loops_0 L^{TN-TL+\eta}$. Up to homotopy, $p_L$  only depends on a choice of orientation on the stable vector bundle. The Thom class can be thought of as a homological version of the coreductions consider in Section~\ref{sec:proof-prop-refpr}. In fact the analogue
\begin{align*}
  p_N \colon \Sigma^{\infty} \loops_0 N_+ \to H\Z
\end{align*}
for $N$ is the coreduction composed with the unit $\Sigma^{\infty}S^0 \to H\Z$.

For any map of spectra $p : X \to H\Z$ we define the composition
\begin{align*}
  \widetilde{p} : X\wedge H\Z \xrightarrow{p\wedge\id} H\Z\wedge H\Z \to H\Z,
\end{align*}
where the last map is the product on $H\Z$, which induce the cup product on cohomology. One recovers the Thom isomorphism by composing the coalgebra map smash the identity on $H\Z$ with $\id \wedge \widetilde{p_L}$:
\begin{align*}
  \loops_0 L^{TN-TL+\eta} \wedge H\Z \to \Sigma^{\infty}\loops_{0}L_+ \wedge \loops_0 L^{TN-TL+\eta} \wedge H\Z \to \Sigma^{\infty}\loops_{0}L_+ \wedge H\Z
\end{align*}

By smashing all entries in the diagram in Equation~\eqref{eq:commutative_diagram} with $H\Z$ from the right (and morphisms with the identity on $H\Z$) we may extend the diagram by composing with the map $\id \wedge \widetilde{p_L}$ and the analogue for $N$ to get the diagram
\begin{align} \label{eq:7}
   \pare*{\vcenter{\xymatrix{ 
    \Sigma^{\infty} \loops_{0}  N_+  \ar[rr]^{\loops_0 j_{!}}\ar[rr]\ar[d] &&
    \loops_{0} L^{TN-TL + \eta} \ar[d] \\ 
    \loops_{0} N_+ \wedge \Sigma^{\infty} \loops_{0} N_+ 
    \ar[r]^{\id\wedge\loops_0 j_{!}} \ar[d]^{\id\wedge \widetilde{p_N}} &
    \loops_{0} N_+ \wedge\loops_{0} L^{TN-TL + \eta} 
    \ar[d]^{\id \wedge \widetilde{p_L}} &
    \loops_{0} L_+\wedge\loops_{0} L^{TN-TL+\eta}
    \ar[l]_-{\loops_0 j\wedge\id} \ar[d]^{\id \wedge \widetilde{p_L}} \\
    \Sigma^{\infty} \loops_{0} N_+ \ar[r]^{\id} &
    \Sigma^{\infty} \loops_{0} N_+ &
    \Sigma^{\infty} \loops_{0} L_+ \ar[l]_-{\Sigma^{\infty}\loops_0 j}
  }}} \wedge H\Z.
\end{align}
Here we have slightly abused notation since two of the lower
vertical maps are only defined after smashing with $H\Z$.
Since the top left vertical map is induced by the diagonal, the
vertical composition from the top left corner to the bottom left is
in fact the identity. The map from the top right corner to the bottom right is $\tau_L$, and the bottom right
rectangle commutes for obvious reasons. To see that the bottom left rectangle commutes, it suffices to prove the commutativity of the diagram
\begin{align} \label{eq:12}
  \xymatrix@C=5em{
    \Sigma^{\infty} \loops_{0} N_+ \wedge H\Z \ar[r]^-{\loops_0 j_! \wedge \id_{H\Z}} \ar[d]^-{p_N} &
    \loops_{0} L^{TN-TL+\eta} \wedge H\Z \ar[d]^-{p_L} \\
    H\Z \ar[r]^-{\id} &  H\Z.
  }
\end{align}
By  \cite[Corollary 8.3]{kragh-transfer},  $\loops_0 j_!$, extends the  classical transfer map $j_! \colon
\Sigma^{\infty} N_+ \to L^{TN-TL}$ obtained by decomposing the total space of a trivial bundle over $N$ of sufficiently high rank as a direct sum of the tangent space of $N$ with a normal bundle, and collapsing onto a neighbourhood of $L$. Passing to homotopy groups, we see that the top map in Equation \eqref{eq:12} is the map induced by $\loops_0 j_! $ on homology. The diagram commutes if we pick the Thom classes (orientations) compatibly with each other. To see that this can be done, note that the composition of $j$ with the projection $T^*N \to N$ is of degree one, so the map on $H_0$ is therefore an isomorphism. We can therefore pick the orientation inducing $p_L$ so that the diagram commutes.

It now follows that since $\tau_L$ and $\Sigma^{\infty}\loops_0 j \wedge \id_{H\Z}$ are equivalences so is $\loops_0 j_! \wedge
\id_{H\Z}$, which in turn by the Hurewitz isomorphism implies that $\loops_0 j_!$ itself is an equivalence.

\end{document}